\documentclass[a4paper,12pt,oneside]{amsart} 

\usepackage[english]{babel}
\usepackage[latin9]{inputenc} 
\usepackage{amsmath,amsthm,amssymb,amsfonts,amscd,amsbsy,amsxtra}
\usepackage{ucs}
\usepackage{graphicx}
\usepackage{paralist}
\usepackage{hyperref}

\newcommand{\R}{\mathbb{R}}

\newcommand{\N}{\mathbb{N}}
\newcommand{\C}{\mathbb{C}}
\newcommand{\Cci}{C_{c}^{\infty}}

\newcommand{\ssubset}{\subset\joinrel\subset}

\usepackage{scalerel}[2014/03/10]
\usepackage[usestackEOL]{stackengine}
\def\avgint{\,\ThisStyle{\ensurestackMath{%
  \stackinset{c}{.2\LMpt}{c}{.5\LMpt}{\SavedStyle-}{\SavedStyle\phantom{\int}}}%
  \setbox0=\hbox{$\SavedStyle\int\,$}\kern-\wd0}\int}

\DeclareMathOperator{\Real}{Re}
\DeclareMathOperator{\Imag}{Im}
\DeclareMathOperator{\dist}{dist}
\DeclareMathOperator{\supp}{supp}
\DeclareMathOperator{\spanup}{span}

\theoremstyle{plain}
\newtheorem{theorem}{Theorem}[section]
\newtheorem{proposition}[theorem]{Proposition}
\newtheorem{corollary}[theorem]{Corollary}
\newtheorem{lemma}[theorem]{Lemma}
\theoremstyle{definition}

\title[Elliptic problems and holomorphic functions]{Elliptic problems and holomorphic functions in Banach spaces}
 
\author[W. Arendt]{Wolfgang Arendt}
\address{Wolfgang Arendt\\Institute of Applied Analysis\\Ulm University\\89069 Ulm\\Germany}
\email{wolfgang.arendt@uni-ulm.de}

\author[M. Bernhard]{Manuel Bernhard}
\address{Manuel Bernhard\\Institute of Applied Analysis\\Ulm University\\89069 Ulm\\Germany}
\email{manuel.bernhard@uni-ulm.de}

\author[M. Kreuter]{Marcel Kreuter}
\address{Marcel Kreuter\\Institute of Applied Analysis\\Ulm University\\89069 Ulm\\Germany}
\email{marcel.kreuter@alumni.uni-ulm.de}

\date{\today}


\begin{document}
\maketitle

\begin{abstract}
In the first part we show that a vector-valued almost separably valued function $f$ is holomorphic (harmonic) if and only if it is dominated by an $L^1_\textnormal{loc}$ function and there exists a separating set $W\subset X'$ such that $\langle f,x'\rangle$ is holomorphic (harmonic) for all $x'\in W$. This improves a known result which requires $f$ to be locally bounded. In the second part we consider classical results in the $L^p$ theory for elliptic differential operators of second order. In the vector-valued setting these results are shown to be equivalent to the UMD property. 
\end{abstract}


\section{Introduction}

Let $f:\Omega\rightarrow X$, where $\Omega$ is an open subset of $\C$ (or $\R^d$) and $X$ is a complex (real) Banach space. The function $f$ is called \emph{holomorphic (harmonic)} if it is complex differentiable (twice partially differentiable with $\Delta f=0$). The first part of this article is concerned with a criterion for vector-valued holomorphy (harmonicity). The function $f$ is called \emph{weakly holomorphic (weakly harmonic)} if $x'\circ f$ is holomorphic (harmonic) for all $x'\in X'$. We say that $f$ is \emph{very weakly holomorphic (very weakly harmonic)} if there exists a separating subset $W\subset X'$ such that $x'\circ f$ is holomorphic (harmonic) for all $x'\in W$.\\

It was shown in \cite{grosseborelokada} (see also \cite{GrosseCriterion}) that a vector-valued function $f$ is holomorphic if and only if it is locally bounded and very weakly holomorphic. This answered a question posted in \cite{WrobelAnalyticFunctions} ten years earlier. A very short proof was given in \cite{ArendtNikolskiHolomorphic}. In \cite{ArendtHarmonic} it was shown that a similar approach yields the analogous result for harmonic functions. The first part of this paper is concerned with an improvement of these results. It is known, that very weak holomorphy alone is not sufficient \cite[Theorem 1.5]{ArendtNikolskiHolomorphic}. However, we will show that the boundedness assumption can be weakened. We say that a set $\mathcal{F}$ of functions from $\Omega$ to $X$ is \emph{locally $L^1$-bounded} if there exists a function $g\in L^1_\textnormal{loc}(\Omega,\R)$ such that $\|f(\xi)\|_X\leq g(\xi)$ holds for almost all $\xi\in\Omega$ and for all $f\in\mathcal{F}$. A single function $f:\Omega\rightarrow X$ is called \emph{locally $L^1$-bounded} if $\{f\}$ is locally $L^1$-bounded. A net $\{f_i\}_{i\in I}$ is called \emph{locally $L^1$-bounded} if $\{f_i,i\in I\}$ is locally $L^1$-bounded. Our result is the following.

\begin{theorem}\label{very weakly holomorphic harmonic}
Let $X$ be a Banach space and $\Omega\subset\C$ ($\Omega\subset\R^d$) be open. A function $f:\Omega\rightarrow X$ is holomorphic (harmonic) if and only if it is locally $L^1$-bounded and very weakly holomorphic (very weakly harmonic).
\end{theorem}

We give two proofs of this result, one of which is very short, but is only valid if $X$ is separable, and one that follows the approach in \cite{ArendtNikolskiHolomorphic} and \cite{ArendtHarmonic}. The first proof will also yield a shortcut proof for the vector-valued version of Weyl's lemma. This result will be needed in the second part of the paper. The set $\mathcal{L}(\mathcal{D}(\Omega,R),X)$ of $X$-valued distributions on $\Omega$ will be denoted by $\mathcal{D}'(\Omega,X)$.

\begin{theorem}[Weyl]\label{Weyl}
Let $\Omega\subset\R^d$ be open, $X$ be a Banach space and let $f\in L^1_\textnormal{loc}(\Omega,X)$ such that $\Delta f=0$ in $\mathcal{D}'(\Omega,X)$. Then $f$ has a harmonic representative; that is, there exists $f^\ast\in C^\infty(\Omega,X)$ such that $\Delta f^\ast=0$ and $f=f^\ast$ almost everywhere. 
\end{theorem}

Recall that $\Delta f=0$ in the sense of distributions means that $\int{f\Delta\varphi}=0$ for all test functions $\varphi\in\Cci(\Omega,\R)$. In view of the formulation of Theorem \ref{very weakly holomorphic harmonic} we want to remark that this is equivalent to saying that there exists a separating set $W\subset X'$ such that $\Delta(x'\circ f)=0$ in $\mathcal{D}'(\Omega,\R)$ for all $x'\in W$.\\

In the second part of the paper, Sections 4 - 6, we investigate some classical elliptic problems. Again, we ask whether the solutions with values in a Banach space have the same regularity as in the scalar case. Using a result by Geiss, Montgomery and Saksman \cite{GeissSingularIntegral} on homogeneous vector-valued multipliers we prove our main technical tool, Thereom \ref{necessity on domains}, on regularity properties of Newtonian potentials. This result will be used to determine the domain of the Laplacian on $L^p(\R^d,X)$. One of our main results shows that on a bounded domain $\Omega$ of class $C^{1,1}$ the following classical property characterizes UMD-spaces: Given $f\in L^p(\Omega,X)$ there exists a unique $u\in W^{1,p}_0(\Omega,X)\cap W^{2,p}(\Omega,X)$ solving $\Delta u=f$. More general elliptic operators are also considered in Section 6.\\

Parts of this work are contained in the third author's thesis \cite{KreuterVectorValued}.


\section{Harmonic and Holomorphic Functions -- The Separable Case}

In this section we will give a proof of Theorem \ref{very weakly holomorphic harmonic} in the case of a separable Banach space.

\begin{lemma}\label{harmonic representative real}
Let $\Omega\subset\R^d$ be open and let $f\in L^1_\textnormal{loc}(\Omega,\R)$ such that $\Delta f=0$ in $\mathcal{D}'(\Omega,\R)$. Furthermore let $\rho_r$ be a mollifier supported in $B(0,r)$, $r>0$, consisting of radial functions. Then $\rho_r\ast f=f$ almost everywhere in $\Omega_r:=\{\xi\in\Omega,\dist(\xi,\partial\Omega)>r\}$.
\end{lemma}

\begin{proof}
Since $\Delta f=0$ distributionally there exists a harmonic representative $f^\ast$ of $f$ by Weyl's Lemma in the real-valued case \cite[Chapter II, \S3, Proposition 1]{DautrayLionsMathematicalAnalysis}. Since $f=f^\ast$ almost everywhere it follows that $\rho_r\ast f=\rho_r\ast f^\ast$ everywhere in $\Omega$. Now by \cite[Chapter 2, Proof of Theorem 6]{Evanspde} we have that $\rho_r\ast f^\ast=f^\ast$ on $\Omega_r$ from which the claim follows.
\end{proof}

\begin{proof}[Proof of Theorem \ref{very weakly holomorphic harmonic} if $X$ is separable]
(a) We start with the case that $f$ is very weakly harmonic. Then $f$ is very weakly measurable, and it follows from the Krein-\v{S}mulyan theorem (c.f. Section \ref{general proof}) that $f$ is measurable -- for a full proof we refer the reader to \cite[Theorem 1.2]{ArendtDegenerate} or \cite[Theorem 1.1.20]{HytonenetalAnalysisinBanachspaces}. Since $f$ is locally $L^1$-bounded it follows that $f\in L^1_\textnormal{loc}(\Omega,X)$. Let $\rho_r$ be a mollifier supported in $B(0,r)$. Then the function $f_r:=\rho_r\ast f$ is well-defined and smooth in $\Omega_r$. Lemma \ref{harmonic representative real} shows that $\langle f_r,x'\rangle=\rho_r\ast\langle f,x'\rangle=\langle f,x'\rangle$ in $\Omega_r$ for every $x'$ in the separating set $W\subset X'$ for which $\langle f,x'\rangle$ is harmonic. Since $W$ is separating it follows that $f_r=f$ in $\Omega_r$. In particular: $f$ is smooth and hence -- using again that $W$ is separating -- $\Delta f=0$.\\

(b) Now we come to the case where $f$ is very weakly holomorphic. Analogously to the harmonic case one sees that $f$ is locally integrable. Let $z_0\in\Omega$ and let $r_0>0$ such that $B(z_0,r_0)\ssubset\Omega$. Since $f$ is integrable on $B(z_0,r_0)$ it follows from Fubini's theorem that $f$ is integrable on the sphere $S(z_0,r)$ for almost all $r\leq r_0$. Choose such an $r$ and define
\begin{align*}
u(z):=\frac{1}{2\pi i}\int_{|w-z_0|=r}{\frac{f(w)}{z-w}\,dw}
\end{align*}
for all $z\in B(z_0,r)$. As in the scalar case one shows that $u$ defines a holomorphic function. Cauchy's integral formula shows that
\begin{align*}
\langle u(z),x'\rangle=\langle f(z),x'\rangle
\end{align*}
for all $x'\in W$ and all $z\in B(z_0,r)$. Since $W$ is separating it follows that $u=f$ and hence $f$ is holomorphic.
\end{proof}

The approach used for the case where $f$ is very weakly harmonic also yields the

\begin{proof}[Proof of Theorem \ref{Weyl}]
Let $\rho_r$ be a mollifier supported in $B(0,r)\ssubset\Omega$. By assumption the function $f_r:=\rho_r\ast f$ is well-defined. Since $f$ is measurable we may assume that $X$ is separable. In this case there exists a countable separating set $W\subset X'$ \cite[Proposition B.1.10]{HytonenetalAnalysisinBanachspaces}. Lemma \ref{harmonic representative real} shows that for every $x'\in W$ there exists a negligible set $N_{x'}$ such that $\langle f_r,x'\rangle=\langle f,x'\rangle$ in $\Omega_r\backslash N_{x'}$. Since $W$ is countable the set $N:=\bigcup_{x'\in W} N_{x'}$ is negligible. Furthermore $W$ separates $X$ and hence $f_r=f$ almost everywhere in $\Omega_r$. For every $x'\in W$ the function $\langle f,x'\rangle$ has a harmonic representative by Weyl's Lemma in the real-valued case \cite[Chapter II, \S3, Proposition 1]{DautrayLionsMathematicalAnalysis}. Since $\langle f_r,x'\rangle$ is a continuous representative of $\langle f,x'\rangle$ in $\Omega_r$ it follows that $\langle f_r,x'\rangle$ is the harmonic representative of $\langle f,x'\rangle$ in $\Omega_r$. Since $W$ is separating it follows from Theorem \ref{very weakly holomorphic harmonic} that $f_r$ is harmonic in $\Omega_r$. The claim now follows by taking a sequence $r_n\rightarrow0$ and defining $f^\ast(\xi):=f_{r_n}(\xi)$, where $\xi\in\Omega_{r_n}$. Then $f^\ast:\Omega\rightarrow X$ is well-defined, harmonic, and conincides with $f$ almost everywhere.
\end{proof}

We want to give a holomorphic version of Theorem \ref{Weyl} using the distributional Cauchy-Riemann equations
\begin{align*}
D_1u&=D_2v\\
D_2u&=-D_2v.
\end{align*}
For vector-valued functions -- which do not have a real or imaginary part -- we make sense of these equations by saying that a function $f:\Omega\rightarrow X$ from an open subset of $\C$ into a complex Banach space satisfies the Cauchy-Riemann equations \emph{very weakly distributionally} if there exists a separating set $W\subset X'$ such that the functions
\begin{align*}
u:=\Real\langle f,x'\rangle\quad\textnormal{and}\quad v:=\Imag\langle f,x'\rangle
\end{align*}
satisfy the Cauchy-Riemann equations distributionally for every $x'\in W$. The following lemma is known, but we present a proof using our results.

\begin{lemma}[{\cite[Theorem 9]{GrayMorrisCauchyRiemann}}]\label{Weyl holomorphic}
Let $\Omega\subset\C$ be open and let $f\in L^1_\textnormal{loc}(\Omega,\C)$ such that $f$ satisfies the Cauchy-Riemann equations distributionally. Then $f$ has a holomorphic representative.
\end{lemma}

\begin{proof}
It follows from the Cauchy-Riemann equations that the functions $u:=\Real f$ and $v:=\Imag f$ are harmonic in $\mathcal{D}'(\Omega,\R)$. Lemma \ref{harmonic representative real} shows that for a radially symmetric mollifier $\rho_r$ supported in $B(0,r)$ we have $f_r:=\rho_r\ast f=f$ almost everywhere in $\Omega_r$. Since $f_r$ is continuously partially differentiable it satisfies the Cauchy-Riemann equations in the classical sense in $\Omega_r$ and thus is holomorphic. We may now define the representative of $f$ analogously to the proof of Theorem \ref{Weyl}.
\end{proof}

\begin{theorem}
Let $\Omega\subset\C$ be open and let $X$ be a complex Banach space. Suppose $f\in L^1_\textnormal{loc}(\Omega,X)$ satisfies the Cauchy-Riemann equations very weakly distributionally. Then $f$ has a holomorphic representative.
\end{theorem}

\begin{proof}
Since $f$ is measurable we may assume that $X$ is separable. By \cite[Theorem B.1.11]{HytonenetalAnalysisinBanachspaces} we may assume that $W$ is countable. Let $\rho_r$ be a radially symmetric mollifier supported in $B(0,r)$ and define $f_r:=\rho_r\ast f$. Let $x'\in W$. By Lemma \ref{Weyl holomorphic} we know that $\langle f,x'\rangle$ has a holomorphic representative and the proof tells us that in $\Omega_r$ this representative is given by $\rho_r\ast\langle f,x'\rangle=\langle f_r,x'\rangle$. Since $W$ is countable, it follows that $f_r$ is a representative of $f$ in $\Omega_r$. Furthermore, Theorem \ref{very weakly holomorphic harmonic} shows that $f_r$ is holomorphic in $\Omega_r$. The representative is then defined as above.
\end{proof}


\section{Harmonic and Holomorphic Functions -- The General Case}\label{general proof}

As announced before, we will now give a proof of Theorem \ref{very weakly holomorphic harmonic} which is valid also in non-seperable spaces. We use arguments of \cite{ArendtNikolskiHolomorphic} and \cite{ArendtHarmonic} but add a new idea to get around with the $L^1_\textnormal{loc}$-hypothesis only. We gather some results which we will need for the proof. By $\sigma_{d-1}$ we denote the $(d-1)$-dimensional Hausdorff measure in $\R^d$ and by $\omega_d$ the Hausdorff measure of the unit sphere $S_{d-1}\subset\R^d$.

\begin{lemma}[{\cite[Lemma 2.3 and Theorem 5.2]{ArendtHarmonic}, also c.f. \cite{GrothendieckVector-valuedholomorphic}}]\label{weakly holomorphic harmonic}
\begin{compactenum}[(a)]
\item Let $\Omega\subset\C$ be open and let $X$ be a Banach space. A function $f:\Omega\rightarrow X$ is holomorphic if and only if it is weakly holomorphic. In this case, $f$ satisfies Cauchy's integral formula
\begin{align*}
f(z)=\frac{1}{2\pi i}\int_{|w-z_0|=r_0}{\frac{f(w)}{w-z}\,dw}
\end{align*}
for all $z_0\in\Omega$ and $r_0>0$ such that $z\in B(z_0,r_0)\ssubset\Omega$.
\item Let $\Omega\subset\R^d$ be open and let $X$ be a Banach space. A function $f:\Omega\rightarrow X$ is harmonic if and only if it is weakly harmonic. In this case, $f$ satisfies Poisson's integral formula
\begin{align*}
f(\xi)=\frac{1}{\omega_d r_0}\int_{S_{d-1}(\xi_0,r_0)}{\frac{r_0^2-|\xi-\xi_0|^2}{|\xi-s|^d}f(s)\,d\sigma_{d-1}(s)}
\end{align*}
for all $\xi_0\in\Omega$ and $r_0>0$ such that $\xi\in B(\xi_0,r_0)\ssubset\Omega$.
\end{compactenum}
\end{lemma}

\begin{proposition}\label{pointwise limit}
Let $X$ be a Banach space and let $\{f_i\}_{i\in I}$ be a locally $L^1$-bounded net of $X$-valued holomorphic (harmonic) functions on the open set $\Omega\subset\C$ ($\Omega\subset\R^d$). Assume that $f:=\lim_{i\in I}f_i$ exists pointwise in $\Omega$. Then $f$ is a holomorphic (harmonic) function and $f=\lim_{i\in I}f_i$ uniformly on compact sets.
\end{proposition}

\begin{proof}
We start with the case of holomorphic functions. Let $z_0\in\Omega$. By Fubini's theorem the net $\{f_i\}_{i\in I}$ is locally $L^1$-bounded on the set $\{w\in\Omega,|w-z_0|=r_0\}$ for almost all $r_0>0$. Fix such an $r_0>0$ and denote by $\Gamma$ the set $\{w\in\Omega,|w-z_0|=r_0\}$. Cauchy's integral formula yields
\begin{align*}
\|f_i(z_1)-f_i(z_2)\|_X&\leq\frac{1}{2\pi}\int_{\Gamma}{\|f_i(w)\|_X\left|\frac{z_1-z_2}{(w-z_1)(w-z_2)}\right|\,dw}\\
&\leq\frac{C}{2\pi}\|f_i\|_{L^1(\Gamma,X)}|z_1-z_2|
\end{align*}
for all $z_1,z_2\in B(z_0,r_0)$, some constant $C=C(\dist(z_1,\Gamma),\dist(z_2,\Gamma))$ and all $i\in I$. Since the net $\{f_i\}_{i\in I}$ is locally $L^1$-bounded this shows that $\{f_i\}_{i\in I}$ is equicontinuous on compact subsets. Hence $f=\lim_{i\in I}f_i$ exists uniformly on compact sets and it follows that $f$ satisfies Cauchy's integral formula and is thus holomorphic. The case of harmonic functions is treated analogously using Poisson's integral formula.
\end{proof}

\begin{theorem}[{Krein-\v{S}mulyan \cite[Corollary 2.7.12]{MegginsonBanachspacetheory}}]\label{KreinSmulyan}
Let $X$ be a Banach space and let $Y\subset X'$ be a subspace. Then $Y$ is closed in the weak-$\ast$ topology if and only if $Y\cap B_{X'}$ is weakly-$\ast$ closed, where $B_{X'}$ denotes the closed unit ball in $X'$.
\end{theorem}

\begin{proof}[Proof of Theorem \ref{very weakly holomorphic harmonic} for general Banach spaces]
Consider the space
\begin{align*}
Y:=\{x'\in X',x'\circ f\textnormal{ is holomorphic (harmonic)}\}.
\end{align*}
Since $W\subset Y$ it follows that $Y$ is weak-$\ast$ dense in $X'$. It remains to show that $Y$ is closed in the weak-$\ast$ topology since then the result follows from Lemma \ref{weakly holomorphic harmonic}. By the Krein-\v{S}mulyan theorem it suffices to show that $Y\cap B_{X'}$ is weakly-$\ast$ closed for every $r>0$. Let $\{x_i'\}_{i\in I}$ be a net in $Y\cap B_{X'}$ such that $x_i'\rightharpoonup^*x'\in B_{X'}$. The net formed by the functions $f_i:=x_i'\circ f$ is locally $L^1$-bounded and converges pointwise to $x'\circ f$. By Proposition \ref{pointwise limit} it follows that $x'\circ f$ is holomorphic (harmonic) and hence $x'\in Y$.
\end{proof}

Vitali's convergence theorem is usually stated for bounded sequences of holomorphic functions. We apply our results to show that it also holds for locally $L^1$-bounded sequences. Let $\Omega$ be an open and connected set in $\C$ (or $\R^d$). A subset $N\subset\Omega$ is called a \emph{set of uniqueness for holomorphic (harmonic) functions} if every holomorphic (harmonic) function which vanishes on $N$ also vanishes on $\Omega$. It is well known that any infinite set contained in a compact subset of$\Omega$ is a set of uniqueness for holomorphic functions. This does not hold for harmonic functions. On the other hand, if the closure of $N\subset\Omega$ has non-empty interior, then $N$ is a set of uniqueness for harmonic functions.

\begin{theorem}[Vitali]
Let $X$ be a Banach space and let $f_n$ be a locally $L^1$-bounded sequence of $X$-valued holomorphic (harmonic) functions. Suppose that $N\subset\Omega$ is a set of uniqueness for holomorphic (harmonic) functions such that $\lim_{n\rightarrow\infty}f_n$ exists pointwise on $N$. Then $\lim_{n\rightarrow\infty}f_n$ exists uniformly on compact sets and defines a holomorphic (harmonic) function.
\end{theorem}

\begin{proof}
The function
\begin{align*}
F:\Omega&\rightarrow\ell^\infty(\N,X)\\
z&\mapsto (f_n(z))_{n\in\N}
\end{align*}
is holomorphic (harmonic) by Theorem \ref{very weakly holomorphic harmonic}. Let $c(\N,X)\subset\ell^\infty(\N,X)$ be the closed subspace of all convergent sequences and denote by $q$ the quotient map $\ell^\infty(\N,X)\rightarrow\ell^\infty(\N,X)\slash c(\N,X)$. Then $q\circ F$ is holomorphic (harmonic) and vanishes on $N$. Since $N$ is a set of uniqueness we have $q\circ F=0$, that is, $F(z)$ is convergent for every $z\in\Omega$. The claim now follows from Proposition \ref{pointwise limit}.
\end{proof}


\section{Newtonian Potentials}

With this section we start the second part of this paper on elliptic $L^p$ theory in Banach spaces. Our results about harmonic functions from Section~\ref{general proof} will play a role in Section \ref{strong solutions section}. 
In the remainder of the paper $X$ denotes a real Banach space.
We recall some facts about the Newtonian potential which can be proved analogously to the real-valued case, see \cite[Section 4.2]{GilbargTrudingerEllipticpde} and \cite[Chapter II, \S3]{DautrayLionsMathematicalAnalysis}. The \emph{fundamental solution} for the Laplace equation is given by
\begin{align*}
\R^d\backslash\{0\}\ni\xi\mapsto\Phi(\xi)=
\begin{cases}
\frac{1}{2\pi}\log|\xi|,&\textnormal{ if }d=2\\
\frac{1}{d(2-d)\lambda(B(0,1))}|\xi|^{2-d},&\textnormal{ if }d>2.
\end{cases}
\end{align*}

For $f\in L^1(\R^d,X)$ with compact support we define the \emph{Newtonian potential} of $f$ via
\begin{align*}
\Phi\ast f.
\end{align*}
The Newtonian potential of $f$ is an element of $L^1_\textnormal{loc}(\R^d,X)$ and satisfies Poisson's equation $\Delta(\Phi\ast f)=f$ in $\mathcal{D}'(\R^d,X)$. Furthermore, if $f$ is compactly supported and H\"older continuous, the Newtonian potential of $f$ is in $C^2(\R^d,X)$ and satisfies $\Delta(\Phi\ast f)=f$ in the classical sense, cf. \cite[Section 4.2]{GilbargTrudingerEllipticpde}.\\

In this section we will show that certain classical $L^p$ estimates for the Newtonian potential on domains imply the UMD property of $X$. For an overview concerning the UMD property we refer the reader to \cite[Chapter 5]{HytonenetalAnalysisinBanachspaces}. The base for our results is the following multiplier theorem. We denote by $\mathfrak{M}L^p(\R^d,X)$ the space of all scalar-valued $L^p(\R^d,X)$ multipliers, see \cite[Definition 5.3.3]{HytonenetalAnalysisinBanachspaces}.

\begin{theorem}[{\cite[Theorem 3.1]{GeissSingularIntegral}}]\label{geissmultiplier}
Let $d\geq 2$ and let $X$ be a Banach space. Let $m\in C^\infty(\R^d\backslash\{0\},\R)$ be even, not constant and \emph{$0$-homogeneous}\index{homogeneous function}, that is,
\begin{align*}
m(\lambda\xi)=m(\xi)
\end{align*}
for all $\xi\in\R^d\backslash\{0\}$ and $\lambda>0$. Suppose that $m\in\mathfrak{M}L^p(\R^d,X)$ for some $1<p<\infty$. Then $X$ has the UMD property
\end{theorem}

\begin{corollary}\label{geisscorollary}
Let $j,k\in\{1,\ldots,d\}$. If the second-order Riesz transform $R_jR_k$ (associated with the multiplier $-\frac{\xi_i\xi_j}{|\xi|^2}$) is bounded in $L^p(\R^d,X)$ for some $1<p<\infty$, then $X$ has the UMD property.
\end{corollary}

Using this corollary we may prove the main result of this section -- a further characterization of the UMD property which will be useful in the next section.

\begin{theorem}\label{necessity on domains}
Let $d\geq2$, $\Omega\subset\R^d$ be open and non-empty and let $1<p<\infty$. Suppose that there exists a constant $C>0$ and $j,k\in\{1,\ldots,d\}$ such that the estimate
\begin{align*}
\|D_{jk}(\Phi\ast f)\|_{L^p(\Omega,X)}\leq C\|f\|_{L^p(\Omega,X)}
\end{align*}
holds for all $f\in\Cci(\Omega,X)$. Then $X$ has the UMD property.
\end{theorem}

\begin{proof}
Since $\Phi\in L^1_\textnormal{loc}(\R^d,\R)$, one has $\Phi\ast f\in C^\infty(\R^d,X)$ for all $f\in\Cci(\R^d,X)$. Moreover
\begin{align*}
\Delta(\Phi\ast f)=f
\end{align*}
in the classical sense. For $f:\R^d\rightarrow X$ and $\lambda>0$ define the dilation
\begin{align*}
f_\lambda(x):=f(\lambda x)
\end{align*}
of $f$. Consider the operator $T:\Cci(\R^d,X)\rightarrow C^\infty(\R^d,X)$ given by
\begin{align*}
Tf=D_{jk}(\Phi\ast f).
\end{align*}
It is remarkable that $T$ commutes with dilation, that is,
\begin{align*}
(Tf)_\lambda=Tf_\lambda
\end{align*}
for all $f\in\Cci(\R^d,X)$. To see this we first note that
\begin{align*}
\Phi(\lambda^{-1}\cdot)=\lambda^{d-2}\Phi(\cdot)+c_d(\lambda)
\end{align*}
where $c_d(\lambda)=0$ if $d>2$ and $c_2(\lambda)$ is a constant. Consequently, for $f\in\Cci(\R^d,X),\lambda>0$ and $\xi\in\R^d$ we have
\begin{align*}
(\Phi\ast f_\lambda)(\xi)&=\int{\Phi(\xi-\eta)f(\lambda\eta)\,d\eta}\\
&=\lambda^{-d}\int{\Phi\left(\xi-\frac{\omega}{\lambda}\right)f(\omega)\,d\omega}\\
&=\lambda^{-d}\int{\Phi(\lambda^{-1}(\lambda\xi-\omega)f(\omega)\,d\omega}\\
&=\lambda^{-2}\int{\Phi(\lambda\xi-\omega)f(\omega)\,d\omega}+\lambda^{-d}c_d(\lambda)\int{f(\omega)\,d\omega}\\
&=\lambda^{-2}(\Phi\ast f)(\lambda\xi)+\lambda^{-d}c_d(\lambda)\int{f(\omega)\,d\omega}.
\end{align*}
Consequently, since the second term does not depend on $\xi$,
\begin{align}\label{Tlambdacommute}
D_{jk}(\Phi\ast f_\lambda)=(D_{jk}(\Phi\ast f))_\lambda.
\end{align}
Next we note that for each measurable function $g:\R^d\rightarrow X$ and $\lambda>0$  we have
\begin{align}\label{changeofvariables}
\lambda^\frac{d}{p}\|g_\lambda\|_{L^p(\R^d,X)}=\|g\|_{L^p(\R^d,X)}.
\end{align}
Using Corollary \ref{geisscorollary} it remains to show that $R_jR_k$ is bounded on $L^p(\R^d,X)$. Since $\Omega$ is open we may assume that it contains $(-1,1)^d$. Let $f\in\Cci(\R^d,X)$. Since $\Delta(\Phi\ast f)=f$ it follows that $R_jR_kf=D_{jk}(\Phi\ast f)=Tf$. Choose $\lambda>0$ such that $\supp f\subset(-\lambda,\lambda)^d$. Thus $f_\lambda\in\Cci(\Omega,X)$ and we have $R_jR_kf=D_{jk}(\Phi\ast f)$ by \eqref{Tlambdacommute}. Using \eqref{Tlambdacommute} and \eqref{changeofvariables} as well as the assumption we obtain
\begin{align*}
\|R_jR_kf\|_{L^p((-\lambda,\lambda)^d,X)}&=\lambda^\frac{d}{p}\|(R_jR_kf)_\lambda\|_{L^p((-1,1)^d,X)}\\
&\leq\lambda^\frac{d}{p}\|(D_{jk}(\Phi\ast f))_\lambda\|_{L^p(\Omega,X)}\\
&=\lambda^\frac{d}{p}\|D_{jk}(\Phi\ast f_\lambda)\|_{L^p(\Omega,X)}\\
&\leq\lambda^\frac{d}{p}C\|f_\lambda\|_{L^p(\Omega,X)}\\
&=C\|f\|_{L^p(\R^d,X)}.
\end{align*}
Letting $\lambda\rightarrow\infty$, this shows that $R_jR_k$ is bounded.
\end{proof}


\section{The Domain of the Laplacian on $L^p(\R^d,X)$}

Let $X$ be a Banach space and $1<p<\infty$. The operator $\Delta_p$ is defined as the distributional Laplacian with maximal domain in $L^p(\R^d,X)$, that is,
\begin{align*}
D(\Delta_p)&:=\{f\in L^p(\R^d,X),\Delta f\in L^p(\R^d,X)\}\\
\Delta_pf&:=\Delta f.
\end{align*}

It is not difficult to see that $\Delta_p$ is the generator of the Gaussian semigroup, see Proposition \ref{Gaussian semigroup} below. If $X$ has the UMD property, the following estimate is known.

\begin{proposition}[{\cite[Proposition 5.5.4]{HytonenetalAnalysisinBanachspaces}}]\label{Lp estimates Laplacian}
Let $X$ be a Banach space that has the UMD property and let $1\leq p<\infty$. Suppose that $u\in L^p(\R^d,X)$ with $\Delta u\in L^p(\R^d,X)$. Then for $j,k\in\{1,\ldots,d\}$ we have $D_{jk} u\in L^p(\R^d,X)$ and there exists a constant $C\geq0$ such that
\begin{align*}
\|D_{jk}u\|_{L^p(\R^d,X)}\leq C\|\Delta u\|_{L^p(\R^d,X)}.
\end{align*}
In fact: $D_{jk} u$ is given by the \emph{second-order Riesz transform} $R_jR_k\Delta u$.
\end{proposition}

Using this we now show

\begin{proposition}\label{UMD implies domain}
Let $X$ be a Banach space which has the UMD property and let $1<p<\infty$. Then
\begin{align*}
D(\Delta_p)=W^{2,p}(\R^d,X).
\end{align*}
\end{proposition}

We will need the following lemmata for the proof.

\begin{lemma}[{\cite[Lemma 5.5.5]{HytonenetalAnalysisinBanachspaces}}]\label{density graph}
The space $\Cci(\R^d,X)$ is a core for $\Delta_p$ $(1\leq p<\infty)$.
\end{lemma}

\begin{lemma}[Interpolation]\label{interpolation}
Let $1\leq p<\infty$ and let $\Omega\subset\R^d$ be open. For every $\varepsilon>0$ there exists a constant $C_\varepsilon>0$ such that
\begin{align*}
\|\nabla u\|_{L^p(\Omega,X^d)}\leq\varepsilon\|D^2u\|_{L^p(\Omega,X^{d\times d})}+C_\varepsilon\|u\|_{L^p(\Omega,X)},
\end{align*}
for every $u\in W^{2,p}_0(\Omega,X)$. If $\Omega$ has a $C^{1,1}$ boundary such an inequality is also valid in $W^{2,p}(\Omega,X)$.
\end{lemma}

\begin{proof}
This can be proved analogously to the real-valued case \cite[Theorems 7.27 and 7.28]{GilbargTrudingerEllipticpde}. Note that the more elegant proof \cite[Exercise 7.19]{GilbargTrudingerEllipticpde} using the Rellich-Kondrachov theorem does not work in this case since the compact embeddings obviously cannot hold in infinite dimensional spaces.
\end{proof}

\begin{proof}[Proof of Proposition \ref{UMD implies domain}]
The inclusion $"\supseteq"$ is clear. Now let $f\in D(\Delta_p)$ and let $\varphi_n\in\Cci(\R^d,X)$ such that $\varphi_n\rightarrow f$ and $\Delta\varphi_n\rightarrow\Delta f$. By the estimates in Proposition \ref{Lp estimates Laplacian} and Lemma \ref{interpolation} there exists a constant $C>0$ such that
\begin{align*}
\|\varphi_n\|_{W^{2,p}(\R^d,X)}\leq C(\|\varphi_n\|_{L^p(\R^d,X)}+\|\Delta\varphi_n\|_{L^p(\R^d,X)}),
\end{align*}
for all $n\in\N$. This shows that $\varphi_n$ is Cauchy in $W^{2,p}(\R^d,X)$ and hence $f\in W^{2,p}(\R^d,X)$.
\end{proof}

We now want to show the converse of Proposition \ref{UMD implies domain}. We will need

\begin{proposition}\label{Gaussian semigroup}
Let $1\leq p<\infty$. The operator $\Delta_p$ is the generator of the strongly continuous Gaussian semigroup $G$ on $L^p(\R^d,X)$ given by
\begin{align*}
(G(t)f)(\xi):=(4\pi t)^{-\frac{d}{2}}\int_{\R^d}{f(\xi-\eta)\exp\left(-\frac{|\eta|^2}{4t}\right)\,d\eta},
\end{align*}
where $t>0,\xi\in\R^d$ and $f\in L^p(\R^d,X)$.
\end{proposition}

\begin{proof}
The assertion is well-known if $X=\R$ \cite[Example 3.7.6]{ArendtBattyHieberNeubranderVector-valuedlaplacetransformation}. Testing with $x'\in X'$ it follows immedeately that $G$ is a semigroup. The strong continuity of $G$ is also well-known \cite[Lemma 1.3.3]{ArendtBattyHieberNeubranderVector-valuedlaplacetransformation}. Let $A$ be the generator of $G$ and let $\Delta_p^\R$ be the operator $\Delta_p$ for $X=\R$. Consider the space
\begin{align*}
D:=\spanup\{f\otimes x,f\in D(\Delta_p^\R),x\in X\}.
\end{align*}
Since $D$ is dense in $L^p(\R^d,X)$ and invariant under the semigroup $G$ it follows that $D$ is a core for $A$ \cite[Proposition I.1.7]{EngelNagelSemigroups}. Obviously, $D\subset D(\Delta_p)$ and $\Delta_p$ coincides with $A$ on $D$. Since $\Delta_p$ is closed, it follows that $A\subset\Delta_p$. To show the inclusion $A\supset\Delta_p$ note that $\lambda\in\rho(A)$ for $\lambda>0$. It remains to show that $\lambda-\Delta_p$ is injective. But this follows immediately from the real-valued case.
\end{proof}

\begin{theorem}
The Banach space $X$ has the UMD property if and only if
\begin{align*}
D(\Delta_p)=W^{2,p}(\R^d,X)
\end{align*}
for some, equivalently all, $1<p<\infty$.
\end{theorem}

\begin{proof}
It remains to show the "if" part. Since $\Delta_p$ generates a $C_0$ semigroup there exists $\mu>0$ such that $\mu\in\rho(\Delta_p)$. By assumption we have
\begin{align*}
\|D_jD_kf\|_{L^p(\R^d,X)}\leq\|f\|_{W^{2,p}(\R^d,X)}\leq C\|\mu f-\Delta f\|_{L^p(\R^d,X)},
\end{align*}
where $C=\|R(\mu,\Delta_p)\|_{\mathcal{L}(L^p(\R^d,X),W^{2,p}(\R^d,X))}$ and $1\leq j,k\leq d$. This holds in particular for $f\in\Cci(\R^d,X)$. Taking the Fourier transform on both sides yields that the function
\begin{align*}
m(\xi):=\frac{-4\pi^2\xi_j\xi_k}{\mu+4\pi^2|\xi|}\quad(\xi\in\R^d)
\end{align*}
is in $\mathfrak{M}L^p(\R^d,X)$. We now use a scaling argument with the same notation as in the proof of Theorem \ref{necessity on domains}. The transformation formula shows that
\begin{align*}
\mathcal{F}^{\pm1}f_\lambda=\lambda^{-d}(\mathcal{F}^{\pm1}f)_{\lambda^{-1}}
\end{align*}
for every $f\in\Cci(\R^d,X)$. Hence the operator $T_{m_\lambda}$ assiociated with the multiplier $m_\lambda$ satisfies
\begin{align*}
T_{m_\lambda}f&=\mathcal{F}^{-1}(m_\lambda\mathcal{F}f)\\
&=\mathcal{F}^{-1}((m(\mathcal{F}f)_{\lambda^{-1}})_\lambda)\\
&=\mathcal{F}^{-1}((m\lambda^d\mathcal{F}f_\lambda)_\lambda)\\
&=(\mathcal{F}^{-1}(m\mathcal{F}f_\lambda))_{\lambda^{-1}}\\
&=(T_mf_\lambda)_{\lambda^{-1}},
\end{align*}
from which we can estimate
\begin{align*}
\|T_{m_\lambda}f\|_{L^p(\R^d,X)}&=\|(T_mf_\lambda)_{\lambda^{-1}}\|_{L^p(\R^d,X)}\\
&=\lambda^{\frac{d}{p}}\|T_mf_\lambda\|_{L^p(\R^d,X)}\\
&\leq\lambda^{\frac{d}{p}}\|T_m\|_{\mathcal{L}(L^p(\R^d,X))}\|f_\lambda\|_{L^p(\R^d,X)}\\
&=\|T_m\|_{\mathcal{L}(L^p(\R^d,X))}\|f\|_{L^p(\R^d,X)}.
\end{align*}
By symmetry we obtain $\|T_{m_\lambda}\|_{\mathcal{L}(L^p(\R^d,X))}=\|T_m\|_{\mathcal{L}(L^p(\R^d,X))}$. Note that
\begin{align*}
m_\lambda\rightarrow-\frac{\xi_j\xi_k}{|\xi|^2}=:m_\infty
\end{align*}
pointwise and that $|m_\lambda|\leq1$ for all $\lambda>0$. By the dominated convergence theorem we have $T_mf\rightarrow T_{m_\infty}f$. Fatou's lemma shows that $m_\infty\in\mathfrak{M}L^p(\R^d,X)$ and hence the claim follows from Theorem \ref{geissmultiplier}.
\end{proof}

\section{Elliptic operators on domains}\label{strong solutions section}

In the last section we showed that the Laplacian on $L^p(\R^d,X)$ has the maximal regularity domain $W^{2,p}(\R^d,X)$ if and only if $X$ is a UMD space. Our aim in this section is to show the analogous result for the Dirichlet Laplacian on a bounded domain $\Omega$ of class $C^{1,1}$. In fact, we also consider more general operators.\\

Let $L$ be an elliptic operator in non-divergence form given by
\begin{align*}
L:=a_{ij}D_{ij}+b_iD_i+c,
\end{align*}
where $a_{ij}, b_i, c\in L^\infty(\Omega,\R)$ and $a=(a_{ij})_{ij}$ is a symmetric matrix satisfying
\begin{align*}
a_{ij}(\cdot)\xi_i\xi_j\geq\lambda|\xi|^2
\end{align*}
almost everywhere in $\Omega$ for some fixed $\lambda>0$ and all $\xi\in\R^d$. In this section we consider the Dirichlet problem
\begin{align*}
\begin{cases}
Lu=f\\
u-\varphi\in W^{1,p}_0(\Omega,X),
\end{cases}
\end{align*}
where $f\in L^p(\Omega,X)$ and $\varphi\in W^{2,p}(\Omega,X)$ are given. We will show that the existence of a unique solution is equivalent to the UMD property. We first start with the sufficiency of the UMD property. For $L$ we have the following $L^p$ estimate.

\begin{theorem}\label{global Lp estimate for L}
Let $\Omega\subset\R^d$ be open and bounded with a $C^{1,1}$ boundary and let $L$ be an elliptic operator as above. Furthermore assume that $a\in C(\overline{\Omega},\R^{d\times d})$. Let $X$ be a Banach space which has the UMD property and let $1<p<\infty$. Then there exists a constant $C>0$ such that
\begin{align*}
\|u\|_{W^{2,p}(\Omega,X)}\leq C(\|u\|_{L^p(\Omega,X)}+\|Lu\|_{L^p(\Omega,X)})
\end{align*}
for all $u\in W^{2,p}(\Omega,X)\cap W^{1,p}_0(\Omega,X)$.
\end{theorem}

\begin{proof}
Proceed as in the proof of \cite[Theorem 9.13]{GilbargTrudingerEllipticpde} proving the estimates for the Laplacian \cite[Theorem 9.9]{GilbargTrudingerEllipticpde} using Proposition \ref{Lp estimates Laplacian} and also using the interpolation estimate in Lemma \ref{interpolation}. 
\end{proof}

To show existence we will need an estimate which does not depend on $\|u\|_{L^p(\Omega,X)}$. As in Lemma \ref{interpolation}, we cannot prove this estimate analogously to the real-valued case \cite[Lemma 9.17]{GilbargTrudingerEllipticpde} since this proof uses the Rellich-Kondrachov theorem. We gather some information about the real-valued case.

\begin{theorem}\label{strong solution existence real and positive}
\begin{compactenum}[(a)]
\item Let $\Omega\subset\R^d$ be open and bounded with a $C^{1,1}$-boundary and let $L$ be an elliptic operator with $a\in C(\overline{\Omega},\R^{d\times d})$ and $c\leq0$. Then for every data $f\in L^p(\Omega,\R)$ and $\varphi\in W^{2,p}(\Omega,\R)$ with $1<p<\infty$ there exists a unique $u\in W^{2,p}(\Omega,\R)$ satisfying $Lu=f$ and $u-\varphi\in W^{1,p}_0(\Omega,\R)$.
\item In the setting of (a) let $\varphi=0$ and define 
\[
    T:L^p(\Omega,\R)\rightarrow W^{2,p}(\Omega,\R)\cap W^{1,p}_0(\Omega,\R)
\]
via $f\mapsto u$. Then $-T$ is a positive operator, that is, $Tf\leq0$ whenever $f\geq0$.
\end{compactenum}
\end{theorem}

\begin{proof}
(a) is the assertion of \cite[Theorem 9.15]{GilbargTrudingerEllipticpde}. For the proof of (b) we first let $f\in\Cci(\Omega,\R)_+$. Then $f\in L^d(\Omega,\R)$ and hence by uniqueness the solution $u:=Tf$ is an element of $W^{2,d}(\Omega,\R)\cap W^{1,d}_0(\Omega,\R)$. Furthermore it is continuous up to the boundary by Morrey's embedding theorem. Since $\Omega$ has a $C^{1,1}$ boundary this implies that $u_{|\partial\Omega}=0$ in the classical sense. Suppose that $u(\xi)>0$ for some $\xi\in\Omega$. Then $u$ has a nonnegative maximum in $\Omega$. This contradicts the maximum principle \cite[Theorem 9.6]{GilbargTrudingerEllipticpde}.\\
Now let $f\geq0$ be arbitrary. There exist nonnegative functions $f_n\in\Cci(\Omega,\R)$ such that $f_n\rightarrow f$ in $L^p(\Omega,\R)$. By the first step we know that the solution $u_n:=Tf_n$ is non-positive. The estimate in \cite[Lemma 9.17]{GilbargTrudingerEllipticpde} shows that $u_n$ is Cauchy and hence convergent in $W^{2,p}(\Omega,\R)\cap W^{1,p}_0(\Omega,\R)$. The uniqueness of the solution shows that $u\leq0$.
\end{proof}

\begin{proposition}\label{estimate without u}
In the setting of Theorem \ref{global Lp estimate for L} let $c\leq0$. Then we have the estimate
\begin{align*}
\|u\|_{W^{2,p}(\Omega,X)}\leq C\|Lu\|_{L^p(\Omega,X)},
\end{align*}
for some $C>0$.
\end{proposition}

\begin{proof}
Let $T$ be the operator in Theorem \ref{strong solution existence real and positive} (b) considered as a bounded operator $L^p(\Omega,\R)\rightarrow L^p(\Omega,\R)$. The operator $T$ can be linearly extended to finite sums of tensors of the form $f\otimes x$ with $f\in L^p(\Omega,\R)$ and $x\in X$. Since $-T$ is a positive operator there exists a unique bounded operator $\tilde{T}$ with the same norm as $T$ mapping $L^p(\Omega,X)\rightarrow L^p(\Omega,X)$ which coincides with $T$ on finite sums of tensors \cite[Theorem 2.1.3]{HytonenetalAnalysisinBanachspaces}. Note that $\tilde{T}Lu=u$ and thus $\|u\|_{L^p(\Omega,X)}\leq\|\tilde{T}\|_{\mathcal{L}(L^p(\Omega,X))}\|Lu\|_{L^p(\Omega,X)}$. Combined with the estimate in Theorem \ref{global Lp estimate for L} this yields the result.
\end{proof}

We are now in a position to prove the existence and uniqueness of strong solutions for the Poisson problem with Dirichlet boundary data.

\begin{theorem}\label{strong solution existence}
Let $\Omega\subset\R^d$ be open and bounded  with a $C^{1,1}$ boundary and let $L$ be an elliptic operator with $a\in C\left(\overline{\Omega},\R^{d\times d}\right)$ and $c\leq0$. Furthermore let $X$ be a space which has the UMD property. Then for every data $f\in L^p(\Omega,X)$ and $\varphi\in W^{2,p}(\Omega,X)$ with $1<p<\infty$ there exists a unique $u\in W^{2,p}(\Omega,X)$ solving $Lu=f$ such that $u-\varphi\in W^{1,p}_0(\Omega,X)$.
\end{theorem}

\begin{proof}
By subtracting $L\varphi$ from $f$ one sees that it is enough to consider the case $\varphi=0$. Let first $f=\sum_{k=1}^{n}f_k\otimes x_k$ be a simple function with $f_k\in L^p(\Omega,\R)$ and $x_k\in X$. For the data $f_k$ the real-valued Theorem \ref{strong solution existence real and positive} yields existence of a solution $u_k\in W^{2,p}(\Omega,\R)\cap W^{1,p}_0(\Omega,\R)$. Thus the function 
\begin{align*}
u:=\sum_{k=1}^{n}u_k\otimes x_k
\end{align*}
is a solution for $f$. For general $f\in L^p(\Omega,X)$ there exists a sequence $f_n$ of finite sums of tensors which converges to $f$ in $L^p(\Omega,X)$. Let $u_n$ be the solution for $f_n$. Then by Proposition \ref{estimate without u} we have
\begin{align*}
\|u_n-u_m\|_{W^{2,p}(\Omega,X)}\leq C\|f_n-f_m\|_{L^p(\Omega,X)}.
\end{align*}
Hence $u_n\rightarrow u$ in $W^{2,p}(\Omega,X)\cap W^{1,p}_0(\Omega,X)$ and $Lu=f$. The uniqueness follows from the $L^p$ estimates.
\end{proof}

The existence theorem has the following converse which gives again a characterization of the UMD property by a regularity property of the Poisson equation on domains.

\begin{corollary}\label{PoissonUMD}
Let $1<p<\infty$, $\Omega\subset\R^d$ be an open and bounded set with a $C^{1,1}$ boundary and let $X$ be a Banach space. The following are equivalent:
\begin{compactenum}[(i)]
\item For every $f\in L^p(\Omega,X)$ there exists a unique $u\in W^{2,p}(\Omega,X)\cap W^{1,p}_0(\Omega,X)$ satisfying $\Delta u=f$.
\item $X$ has the UMD property.
\end{compactenum}
\end{corollary}

\begin{proof}
It remains to show the implication $(i)\Rightarrow(ii)$. Let $\omega\ssubset\Omega$ be nonempty with a $C^{1,1}$ boundary. For $f\in L^p(\omega,X)$ denote by $\tilde{f}$ the extension to $\R^d$ by $0$. Recall that $w:=\Phi\ast\tilde{f}\in L^1_\textnormal{loc}(\R^d,X)$ satisfies $\Delta w=\tilde{f}$ in $\mathcal{D}'(\Omega,X)$.

Let $u$ be the solution for $\tilde{f}$ according to $(i)$. Then $u-w$ solves $\Delta(u-w)=0$. Theorem \ref{Weyl} shows that $u-w$ has a harmonic representative which is in particular in $C^2(\Omega,X)$. Since $u\in W^{2,p}(\omega,X)$ we also have $w=u-(u-w)\in W^{2,p}(\omega,X)$. Hence the Newtonian potential defines a mapping from $L^p(\omega,X)$ into $W^{2,p}(\omega,X)$. We claim that the graph of this mapping is closed. Let $f_n\rightarrow f\in L^p(\omega,X)$ such that $\Phi\ast f_n\rightarrow w\in W^{2,p}(\omega,X)$. Then for every $x'\in X'$ the functions $\langle f_n,x'\rangle,\langle f,x'\rangle$ and $\langle w,x'\rangle$ satisfy the analogue. $\R$ has the UMD property and thus Theorem \ref{Lp estimates Laplacian} shows that $\langle w,x'\rangle=\langle\Phi\ast f,x'\rangle$. Choosing $x'$ from a countable separating subset of $X'$ \cite[Proposition B.1.10]{HytonenetalAnalysisinBanachspaces} yields the claim. Now the closed graph theorem shows the existence of a constant $C\geq0$ such that $\|\Phi\ast f\|_{W^{2,p}(\omega,X)}\leq C\|f\|_{L^p(\omega,X)}$ for all $f\in L^p(\omega,X)$. Finally, Theorem \ref{necessity on domains} shows that $X$ has the UMD property.
\end{proof}

We want to relate Corollary \ref{PoissonUMD} to the generator of the Dirichlet Laplacian. At first we establish an abstract result.

\begin{lemma}
Let $\Omega\subset\R^d$ be an open set, $1\leq p<\infty$ and let $T$ be a positive strongly continuous semigroup on $L^p(\Omega,\R)$ with generator $A$. Let $X$ be a Banach space. Then there exists a unique strongly continuous semigroup $\tilde{T}$ on $L^P(\Omega,X)$ satisfying $x'\circ\tilde{T}(t)f=T(t)(x'\circ f)$ for all $f\in L^p(\Omega,X)$ and all $x'\in X'$. Denote by $\tilde{A}$ its generator. Let $f,g\in L^p(\Omega,X)$. Then $f\in D(\tilde{A})$ and $\tilde{A}f=g$ if and only if $x'\circ f\in D(A)$ and $A(x'\circ f)=x'\circ g$ for all $x'\in X'$.
\end{lemma}

\begin{proof}
By \cite[Theorem 2.1.3]{HytonenetalAnalysisinBanachspaces} there is a unique bounded operator $\tilde{T_t}$ on $L^p(\Omega,X)$ such that $\tilde{T_t}(f\otimes x)=T_tf\otimes x$ for all $f\in L^p(\Omega,\R)$ and all $x\in X$. This is the same as saying that $x'\circ\tilde{T_t}f=T_t(x'\circ f)$ for all $f\in L^p(\Omega,X)$ and $x'\in X'$. It is obvious from the first property that $\tilde{T}:=(\tilde{T_t})_{t\geq0}$ is a strongly continuous semigroup on $L^p(\Omega,X)$. For $f,g\in L^p(\Omega,X)$ one has $f\in D(\tilde{A})$ and $\tilde{A}f=g$ if and only if $\int_0^t{\tilde{T_s}g\,ds}=\tilde{T_t}f-f$ for all $t>0$. Using this and the corresponding assertion for $A$ the last claim follows from the fact that the integral commutes with functionals.
\end{proof}

Now let $\Omega\subset\R^d$ be open and bounded with a $C^{1,1}$ boundary and let $1\leq p<\infty$. Then the operator $A$ given by
\begin{align*}
D(A)&:=W^{1,p}_0(\Omega,\R)\cap W^{2,p}(\Omega,\R)\\
Au&:=\Delta u
\end{align*}
generates a positive strongly continuous semigroup $T$ on $L^p(\Omega,\R)$. Consider the induced semigroup $\tilde{T}$ on $L^p(\Omega,X)$, where $X$ is a Banach space, and denote by $\tilde{A}$ its generator. Then clearly
\begin{align*}
W^{1,p}_0(\Omega,X)\cap W^{2,p}(\Omega,X)\subset D(\tilde{A})
\end{align*}
by the preceding lemma. The identity does not hold in general:

\begin{corollary}
Let $1<p<\infty$ and let $T,A,\tilde{T}$ and $\tilde{A}$ be as above. Then
\begin{align*}
W^{1,p}_0(\Omega,X)\cap W^{2,p}(\Omega,X)=D(\tilde{A})
\end{align*}
if and only if $X$ has the UMD property.
\end{corollary}

\begin{proof}
By \cite[Theorem 2.1.3]{HytonenetalAnalysisinBanachspaces} the norms of $T$ and $\tilde{T}$ coincide. Since $\|T_t\|\leq Me^{-\varepsilon t}$ for all $t>0$ the same estimate holds for $\tilde{T}$. Thus $\tilde{A}$ is invertible. Now Corollary \ref{PoissonUMD} yields the claim.
\end{proof}

\bibliographystyle{abbrv}
\def\cprime{$'$}

\end{document}